\documentclass{article}

\usepackage{libertinus}

\usepackage{amsthm,amsmath,amssymb,latexsym}
\usepackage{hyperref}

\newtheorem{theorem}{Theorem}

\newtheorem{lemma}[theorem]{Lemma}

\newtheorem{remark}[theorem]{Remark}

\newtheorem{example}[theorem]{Example}
\newtheorem{assumption}[theorem]{Assumption}

\def\mE{\mathbb{E}}
\def\mP{\mathbb{P}}

\begin{document}

\title{On stochastic stability of systems\\ driven by red noise\thanks{The authors gratefully acknowledge the support of 
the National Research, Development and Innovation Office (NKFIH) through grants K 143529 and KKP 137490.}}

\author{Bal\'azs Gerencs\'er \and Mikl\'os R\'asonyi\thanks{B.\ Gerencs\'er (email: gerencser.balazs@renyi.hu) and M.\ R\'asonyi (email: rasonyi.miklos@renyi.hu) are with HUN-REN Alfr\'ed R\'enyi Institute of Mathematics, Budapest, Hungary and E\"otv\"os Lor\'and{}
University, Budapest, Hungary }}

\date{\today}

\maketitle
{}

\begin{abstract}
We consider discrete-time non-linear stochastic dynamical systems where the driving noise is 
not i.i.d.\ but an autoregressive process (``red noise''). Under a standard dissipativity condition, 
we prove geometrically fast convergence to 
a stationary distribution in total variation. 
\end{abstract}

\section{Introduction}

We consider a pair of processes $(Y_{n},X_{n})\in\mathbb{R}^d\times\mathbb{R}^d$, $n\in\mathbb{N}$. The first 
is an autoregressive process
\begin{equation}\label{recursion0}
Y_{n+1}-Y_n=-hAY_{n}+\varepsilon_{n+1},
\end{equation}
where $\varepsilon_{n}$, $n\in\mathbb{N}$ is an i.i.d.\ sequence in $\mathbb{R}^{d}$, $h>0$ and $A$ is a matrix.
$Y_n$, $n\geq 1$ will represent the ``red noise'' driving our system. 
The second process is given by the recursion
\begin{equation}\label{recursion1}
X_{n+1}-X_{n}=h \mu(X_n)+\sigma(X_n)Y_{n+1},
\end{equation}
where $\mu:\mathbb{R}^d\to\mathbb{R}^d$, $\sigma:\mathbb{R}^d\to \mathbb{R}^{d\times d}$ are measurable functions. 
We assume throughout this paper that
$(Y_0,X_0)$ are independent of $(\varepsilon_n)_{n\geq 1}$.

\begin{example}{\rm Consider the stochastic differential equations
\begin{eqnarray}
d\hat{Y}_t &=& -A\hat{Y}_t\, dt+dL_t,\label{eho}\\ 
d\hat{X}_t &=& \mu(\hat{X}_t)\, dt+\sigma(\hat{X}_t)d\hat{Y}_t,\label{eho1}
\end{eqnarray}
where $L_t$ is a L\'evy-process. $\hat{Y}_t$ is a specimen of the often-studied Ornstein-Uhlenbeck processes,
see \cite{ou}. In this setting, $(Y_n,X_n)$ corresponds to the Euler-discretization with
step size $h$ for \eqref{eho} and \eqref{eho1}.

However, $(Y_n,X_n)$ is not necessarily of this form. We only assume that, observed along a given time grid,
the system obeys the dynamics \eqref{recursion1} with an autoregressive noise \eqref{recursion0}. In
particular, the law of $\varepsilon_n$ need not be infinitely divisible.
}
\end{example}

Stochastic dynamical systems like \eqref{recursion1} appear frequently in
the literature with a stationary driving noise $Y_{n}$. 
Such $X_n$
are called a \emph{stochastically recursive sequence}, thoroughly studied in \cite{borovkov} but only
under some contractivity assumption on $\mu$. 

In our setting, $Y_n$ is not necessarily stationary, but it is an autoregressive (hence Markovian) process,
called \emph{red noise}. This is a particularly often studied case, motivated by applications
from physics and economics, see \cite{red1,starmer}. More recently, interest in these models has
intensified \cite{red2,red3,red5,red6,red7}.

When $\mu$ is not contractive, only dissipative (see Assumption \ref{equation} below), the techniques of \cite{borovkov}
do not apply and literature becomes scarce.
\cite{hairer-cont} studied stochastic differential equations driven
by fractional Brownian motion, introducing deep techniques.
The methodology of \cite{hairer-cont} was translated to a discrete-time context by \cite{varvenne}.
The main focus of the latter paper is on stationary and Gaussian $Y_{n}$ whose
covariance function may decay slowly (that is, at a power rate, such as the
increments of fractional Brownian motion). Such $Y_n$ fail the Markov property and intricate
arguments are needed to establish convergence of $X_{n}$ to a limiting law.

Subsection 2.6 of \cite{varvenne} also treats a special case of our setting: when $\varepsilon_n$
are standard Gaussian, $A$ is constant times the identity, $Y_n$ is stationary,
$\mu$ and $\sigma$ satisfy continuity hypotheses, \cite{varvenne}
establishes convergence of $X_{n}$ to a limiting law at a rate $O(n^{-p})$, for all $p>0$.

The current paper offers essential improvements with respect to \cite{varvenne}: non-Gaussian $\varepsilon_n$ 
are allowed; our conditions on $\mu,\sigma$ are weaker; 
a more general autoregressive dynamics is considered; $Y_n$ need not be stationary, and, most importantly, we 
prove optimal, exponential convergence
rate of the system to its invariant distribution. We stress that the techniques of \cite{varvenne} apply 
in a much more general context than ours. Our contribution consists of proving sharper 
results in the specific case where $Y_n$ is an autoregressive process. 

Note that, in our
setting, $X_{n}$ fails the Markov property but the pair $(X_{n},Y_{n})$ \emph{is} Markovian.
The textbook approach (see e.g.\ \cite{mt}) is verifying $\psi$-irreducibility and aperiodicity of $(X_n,Y_n)$ and
then showing that the ``Foster-Lyapunov drift condition'' and ``minorization on a small set'' hold.
Although the system \eqref{recursion0}, \eqref{recursion1} satisfy a drift condition
(see Lemma \ref{foster-lyapunov} below), aperiodicity and minorization
seem laborious to verify. For this reason we construct the necessary couplings directly, using a slightly unusual approach.

Section \ref{resu} states our assumptions and main result, Section \ref{proo} provides the proofs.
Section \ref{conclu} is about future research directions.

\section{Convergence to a stationary state}\label{resu}

We are working on a probability space $(\Omega,\mathcal{F},\mathbb{P})$, expectation is denoted by $\mathbb{E}[\cdot]$. For each
event $A\in\mathcal{F}$, ${1}_{A}$ refers to the corresponding indicator function.
Euclidean norm on $\mathbb{R}^m$ is denoted $|\cdot|$, $m$ will always be clear from the context;
$\mathcal{B}(\mathbb{R}^m)$ refers to the Borel sets.
For probabilities $\mu_1,\mu_2$ on $\mathcal{B}(\mathbb{R}^m)$, their
total variation distance is defined as
$$
d_{TV}(\mu_1,\mu_2):=\inf_{X_1\sim \mu_1,X_2\sim \mu_2}\mathbb{P}(X_1\neq X_2)
$$
where the infimum is over pairs of random variables with the respective laws.

For $d\times d$ matrices $M$ we use the norm 
$
||M||:=\sup_{|x|\leq 1}|Mx|,
$
$d$-dimensional Lebesgue measure is denoted $\mathrm{Leb}_d(\cdot)$. $B_d(r)$ denotes the (closed) ball of
radius $r$ around $0\in\mathbb{R}^d$. The standard scalar product in $\mathbb{R}^d$ is denoted $\langle\cdot,\cdot\rangle$.
Our technical 
assumptions we need are listed below.

\begin{assumption}\label{initiale}
$
C_{0}:=\mathbb{E}[|X_0|+|Y_0|]<\infty.
$
\end{assumption}

\begin{assumption}\label{bb}
The random variable $\varepsilon_{0}$ has a law that is absolutely continuous w.r.t.\ $\mathrm{Leb}_d$. Its density function 
$f:\mathbb{R}^{d}\to\mathbb{R}_{+}$ is 
bounded away from $0$ on bounded sets, that is, $\inf_{x\in B_d(r)}|f(x)|\geq \lambda_r>0$ for
each $r>0$.	
$\mathbb{E}[|\varepsilon_{0}|]<\infty$.
\end{assumption}

\begin{assumption}\label{mju}
$\sigma$ is bounded. The matrix $\sigma(x)$ is invertible
for all $x\in\mathbb{R}^d$. The mapping $x\to\sigma^{-1}(x)$ is locally bounded.
\end{assumption}

\begin{assumption}\label{aat}
$A+A^T$ is positive definite.
\end{assumption}

\begin{assumption}\label{equation}   
There are $c_{1},c_{2},c_3>0$ such that
\begin{align}
  \langle \mu(x),x\rangle&\leq -c_{1} |x|^{2}+c_{2},\ x\in\mathbb{R}^{d},\label{kell}\\
  |\mu(x)|&\leq c_3(1+|x|),\ x\in\mathbb{R}^d.\label{valaki}
\end{align}
\end{assumption}

\begin{remark} {\rm Assumptions \ref{initiale}, \ref{bb} are rather mild. Unlike \cite{varvenne}, we do not need continuity of
either $\mu$ or $\sigma$, only Assumption \ref{mju}. The positivity of $A+A^T$ is equivalent to the ergodicity of \eqref{eho} in the continuous-time context, 
see e.g. \cite{risken}. \eqref{kell} is a standard dissipativity condition of Markov chain theory,
appearing also in e.g.\ \cite{deya_et_al,varvenne}, \eqref{valaki} is also a standard requirement
on linear growth.}
\end{remark}

\begin{theorem}\label{main}
Under Assumptions \ref{initiale}, \ref{bb}, \ref{mju}, \ref{aat} and \ref{equation}, there is a probability 
$\Pi_*$ on
$\mathcal{B}(\mathbb{R}^{2d})$ such that 
\begin{equation}\label{totale}
d_{TV}(\mathrm{Law}(X_{n},Y_{n}), \Pi_*)=O(\rho^n)	
\end{equation} for some $\rho<1$. $\Pi_{*}$ is the unique stationary distribution	
of the Markovian system \eqref{recursion0}, \eqref{recursion1} with $\int_{\mathbb{R}^{2d}}|w|\Pi_{*}(dw)<\infty$.
\end{theorem}

\section{Proofs}\label{proo}

This section is divided into four subsections: the first one verifies a Foster-Lyapunov
type drift condition for the system $(X_n,Y_n)$; the second realizes a
coupling that is used when two copies of the system return into a suitable compact
set simultaneously; the third one constructs a copy of the process that will be ``close''
to a shifted version of the same process. 
Finally the proof of Theorem \ref{main} is completed.

\subsection{Lyapunov functions}

\begin{lemma}\label{varv} Under Assumptions \ref{mju} and \ref{equation}, there is $h_0>0$ such that,
for $h\leq h_0$,
there exist $0<\delta_1<1$ and $C_1=C_1(h)>0$ such that
\begin{equation}\label{dollars}
|x+h\mu(x)+\sigma(x)w|\leq \delta_1 |x|+C_1(1+|w|),\ x,w\in\mathbb{R}^d.
\end{equation}
\end{lemma}
\begin{proof}
See Theorem 2.2 of \cite{varvenne,varvenne1}. 
Note that in \cite{varvenne} continuity is assumed for $\sigma,\mu$, but this
property is not used in the proof of \eqref{dollars}.
\end{proof}

\begin{lemma}\label{foster-lyapunov} Let $\mathcal{G}_{n}$, $n\in\mathbb{N}$ be an arbitrary
filtration such that, for all $n\in\mathbb{N}$, 
$\sigma(X_{j},Y_{j},0\leq j\leq n)\subseteq\mathcal{G}_{n}$ and $\varepsilon_{n+1}$ is independent of $\mathcal{G}_{n}$.

    Under Assumptions \ref{mju} and \ref{equation} there exist $h_1>0$, $B\geq 1$ so that for $h\leq h_1$ 
    \[
    V(x,y) = |x|+B|y|
    \]
    serves as a proper Lyapunov function. That is, there exist $C_2>0, \delta_2\in(0,1)$ so that
    \begin{equation}\label{nemo}
    \mE[V(X_{n+1},Y_{n+1})\mid\mathcal{G}_{n}]\leq \delta_2 V(X_{n},Y_{n})+C_2.
    \end{equation}
        Furthermore, there are $C_{3}>0$ and $\delta_2<\delta_3<1$ such that on the event
$\{|X_n|+B|Y_n|>C_3\}$
\begin{equation}
    \label{eq:lyap_contraction}
    \mE_n[V(X_{n+1},Y_{n+1})\mid\mathcal{G}_{n}]\leq \delta_3 [V(X_{n},Y_{n})],
\end{equation}
that is, we have a strict contraction in expectation.
\end{lemma}
\begin{proof}
  Apply Lemma \ref{varv} with $(x,w)=(X_n,Y_{n+1})$, add $B|Y_{n+1}|$ to two both sides ($B\geq 1$ to be specified later):
\begin{equation}\label{blofeld1}
|X_{n+1}|+B|Y_{n+1}|\leq \delta_1 |X_{n}|+(B+C_1)|Y_{n+1}|+C_1.
\end{equation}
Using the definition of the process \eqref{recursion0} we can write
\begin{equation}\label{blofeld2}
\mE[|Y_{n+1}| \mid \mathcal{G}_{n}] \leq \|I-hA\|\cdot |Y_n| + C_{\varepsilon},
\end{equation}
with $C_{\varepsilon}:=\mE[|\varepsilon_0|]$ since $\varepsilon_{n+1}$ is independent of $\mathcal{G}_{n}$ and $Y_{n}$
is $\mathcal{G}_{n}$-measurable.
Taking conditional expectation of \eqref{blofeld1} and using \eqref{blofeld2} we obtain
\begin{equation}
  \label{eq:lyap_contract_proof}
  \begin{aligned}
    \mE[&|X_{n+1}|+B|Y_{n+1}|\mid\mathcal{G}_{n}]\\
    &\leq \delta_1 |X_{n}|+(B+C_1)\|I-hA\|\cdot |Y_n|+C_2,
  \end{aligned}    
\end{equation}
with $C_2 = C_{\varepsilon}(B+C_1)+C_1$, noting that $X_{n}$ is $\mathcal{G}_{n}$-measurable.

We need to control the norm $\|I-hA\|$.
By Assumption \ref{aat} there exists $\gamma>0$ so that $A^\top+A-\gamma I$ is positive definite. For any unit vector $y$,
\begin{align*}
|(I-hA)y|^2&= y^\top(I-hA)^\top(I-hA)y\\
&=|y|^2-hy^\top(A^\top+A)y+h^2y^\top A^\top A y\\
&\leq 1 - h\gamma +h^2\|A\|^2.  
\end{align*}
Therefore for a certain $h_1>0$ this is bounded above by an appropriate $\tilde\delta_2<1$,
for all $0<h\leq h_1$.
Let us now compare the coefficients of the $Y$ terms in \eqref{eq:lyap_contract_proof}.
Choosing $$B\geq \max\left\{1,\frac{2C_{1}\tilde{\delta}_{2}}{1-\tilde{\delta}_{2}}\right\},$$
we see that 
\[
\frac{(B+C_1)\|I-hA\|}{B}\leq \frac{1+\|I-hA\|}{2}\leq\frac{1+\tilde\delta_2}{2}<1,
\]
thus we can further bound the right-hand side of \eqref{eq:lyap_contract_proof} as
\begin{equation*}
    \mE[|X_{n+1}|+B|Y_{n+1}|\mid\mathcal{G}_{n}]\leq \delta_1 |X_{n}|+\frac{1+\tilde\delta_2}{2}B |Y_n|+C_2.
\end{equation*}
This confirms the claim with $\delta_2 = \max(\delta,\frac{1+\tilde\delta_2}{2})$.

Concerning the additional claim, for any $\delta_2<\delta_3<1$ we can satisfy 
$$\delta_2 V(X_{n},Y_{n})+C_2\leq \delta_3 
V(X_{n},Y_{n})$$ 
as soon as $V(X_{n},Y_{n})\geq C_2/(\delta_3-\delta_2)=:C_{3}$.
\end{proof}

Define the filtration $\mathcal{F}_n:=\sigma(X_0,Y_0,\varepsilon_j,1\leq j\leq n)$,
$n\in\mathbb{N}$. Clearly, the above lemma applies with $\mathcal{G}_{n}=\mathcal{F}_{n}$. 
We will, however, use it for a larger filtration later on.

We immediately get a bound globally in time.
\begin{lemma}\label{moments} The following moment estimate holds:
$$
C_5:=\sup_{n\in\mathbb{N}}\mathbb{E}[V(X_n,Y_n)]<\infty.
$$
\end{lemma}
\begin{proof}
Iterating \eqref{nemo} in Lemma \ref{foster-lyapunov} we arrive at
  \vspace{-8pt}
\[
\mathbb{E}[V(X_n,Y_n)]\leq \delta_2^n\mathbb{E}[V(X_0,Y_0)]+\sum_{j=0}^{n-1}\delta_2^j C_2
\vspace{-8pt}
\]
  so by Assumption \ref{initiale}
  \[
C_5\leq \mathbb{E}[V(X_0,Y_0)]+\frac{C_2}{1-\delta_2}\leq BC_{0}
+\frac{C_2}{1-\delta_2}<\infty.\]
\end{proof}

\subsection{Local coupling construction}\label{loco}

Assumptions \ref{initiale}, \ref{bb}, \ref{mju}, \ref{aat} and \ref{equation} will 
be in force till the end of the paper and $h\leq h_{1}$ is fixed.

We explain the underlying idea of our proofs, in a heuristic way.
We will consider two copies $(X_n,Y_n)$ and $(\bar{X}_n,\bar{Y}_n)$ of the same process with
different initializations. We will monitor their return into a compact set. For simplicity, let us
assume that this occurs at time $n$ for both copies.

Our purpose is to construct random variables $(\bar{\varepsilon}_{n+1},\bar{\varepsilon}_{n+2})$
with the same law as $(\varepsilon_{n+1},\varepsilon_{n+2})$ in such a way that updating
the two processes with the respective noises, they should coincide with positive probability.
In what follows, $(x_0,\bar{x}_0,y_0,\bar{y}_0)$ will correspond to the (frozen) initial
values $(X_n,\bar{X}_n,Y_n,\bar{Y}_n)$, $(e_1,e_2)$ are the frozen values of
$(\varepsilon_{n+1},\varepsilon_{n+2})$. Using the dynamics of the system, we will write
and solve equations for finding the appropriate values for $(\bar{\varepsilon}_{n+1},\bar{\varepsilon}_{n+2})$,
denoted $(\bar{e}_1,\bar{e}_2)$, as a function of the frozen parameters. This simple plan
will be executed in two steps.

First define, for  $x_{0},y_{0},\bar{x}_{0},\bar{y}_{0},e_1\in\mathbb{R}^d$,
\begin{align*}
  F(x_{0},\bar{x}_{0},&y_{0},\bar{y}_{0},e_1) :=\\
&\sigma(\bar{x}_{0})^{-1}[x_{0}+h\mu(x_{0})+\sigma(x_{0})[y_{0}-hAy_{0}]
\\
+\, &\sigma(x_{0})e_1- \bar{x}_{0}-h\mu(\bar{x}_{0})-
\sigma(\bar{x}_{0})
[\bar{y}_{0}-hA\bar{y}_{0}]],
\end{align*}
When all other variables are fixed, this function is invertible in $e_{1}$, we denote its inverse by 
$F^{-1}(x_{0},\bar{x}_{0},y_{0},\bar{y}_{0},\cdot)$.
In the heuristic picture explained above, 
choosing $\bar{e}_1$ to be $F(x_{0},\bar{x}_{0},y_{0},\bar{y}_{0},e_1)$ ensures that $X_{n+1}=\bar{X}_{n+1}$.

Assume $K_{0}>0$ to be fixed for the moment, which will be chosen later.
By boundedness of $\sigma$ and local boundedness of $\mu,\sigma^{-1}$, there is a constant $c_\flat(K_{0})$ so that for all $x_{0},\bar{x}_{0},y_{0},\bar{y}_{0}\in B_d(K_{0})$ and $e_1\in B_d(1)$
\begin{equation}\label{dor}
|F(x_{0},\bar{x}_{0},y_{0},\bar{y}_{0},e_1)|\leq c_\flat(K_{0}).
\end{equation}
The derivative of $e_1\!\to\! F(x_{0},\bar{x}_{0},y_{0},\bar{y}_{0},e_1)$ is 
$\sigma^{-1}(\bar{x}_{0})\sigma(x_{0})$.
By boundedness of $\sigma$ and local boundedness of $\sigma^{-1}$, there is a constant $c_\sharp(K_{0})$ so that for all $x_{0},\bar{x}_{0}\in B_d(K_{0})$
\begin{equation}\label{dor1}
|\mathrm{det}(\sigma^{-1}({x}_{0})\sigma(\bar{x}_{0}))|\leq c_\sharp(K_{0}).
\end{equation}

Let $\kappa_{0}(A):=\eta_{0} \mathrm{Leb}(A\cap B_{d}(1))$, $A\in\mathcal{B}(\mathbb{R}^{d})$ 
with some $\eta_{0}>0$ to be chosen
later. Let $\bar{\kappa}_{0}(x_{0},y_{0},\bar{x}_{0},\bar{y}_{0},\cdot)$ be the image of the measure $\kappa_{0}$
by the mapping $F(x_{0},y_{0},\bar{x}_{0},\bar{y}_{0},\cdot)$.{}
By \eqref{dor}, $\bar{\kappa}$ is clearly concentrated on
$B_{d}(c_\flat(K_{0}))$, and by the density transformation formula,
it has a density 
\begin{align*}
  k_{0}(x_{0},&\bar{x}_{0},y_{0},\bar{y}_{0},\bar{e}_1)=\\
  &\frac{\eta_{0} 1_{F(B_d(1))}(\bar{e}_1) }
{\left|\mathrm{det}\left(\nabla F(x_{0},\bar{x}_{0},y_{0},\bar{y}_{0},F^{-1}(x_{0},\bar{x}_{0},y_{0},\bar{y}_{0},\bar{e}_1)
\right)\right|}
\end{align*}
which is bounded from above by $\eta_{0} c_{\sharp}(K_{0})$, by \eqref{dor1} and by the boundedness of $F^{-1}$.
Now choose $\eta_{0}$ so small that
\begin{equation}\label{vo}
\eta_{0} c_{\sharp}(K_{0})\leq \lambda_{c_\flat(K_{0})}/2\leq\inf_{\bar{e}_1\in B_d(c_{\flat}(K_{0}))}f(\bar{e}_1)/2.
\end{equation}

It follows from \eqref{vo} that setting
$$
\upsilon_0(x_{0},\bar{x}_{0},y_{0},\bar{y}_{0}, A) \! := \! \frac{\int_A [f(\bar{e}_1) \!-\! k_{0}(x_{0},\bar{x}_{0},y_{0},\bar{y}_{0},\bar{e}_1)]
\, d\bar{e}_1}{\int_{\mathbb{R}^{d}} [f(\bar{e}_1) \!-\! k_{0}(x_{0},\bar{x}_{0},y_{0},\bar{y}_{0},\bar{e}_1)]\, d\bar{e}_1}
$$
for $A\in\mathcal{B}(\mathbb{R}^{d})$ defines a probabilistic kernel.
Let $(U_{0},\hat{U}_{0})$ be uniform on $[0,1]^{2}$.

\begin{lemma}\label{consu} There is a measurable
$g_0:B_d(K_{0})^4\times [0,1]\to \mathbb{R}^{d}$ such that for all $x_{0},\bar{x}_{0},y_{0},\bar{y}_{0}\in B_{d}(K_{0})$,
$$
\mathrm{Law}(g_0(x_{0},\bar{x}_{0},y_{0},\bar{y}_{0},U_{0}))=\upsilon_0(x_{0},\bar{x}_{0},y_{0},\bar{y}_{0}, \cdot).{}
$$    
\end{lemma} 
\begin{proof} This is folklore, see Lemma 3.3 of \cite{saga-ld}.
\end{proof}

Now we describe, in terms of the heuristic picture, what happens at time $n$:
if $(x_0,y_0)=(\bar{x}_0,\bar{y}_0)$
then take $\bar{\varepsilon}_{n+1}:=\varepsilon_{n+1}$.
If $(x_0,y_0)\neq (\bar{x}_0,\bar{y}_0)$, $x_0,\bar{x}_0,y_0,\bar{y}_0\in B_d(K_{0})$, on the event
$J_{0}:=\{|\varepsilon_{n+1}|\leq 1\}\cap \{\hat{U}_{0}\leq \eta_{0}\}$ 
we define 
$$\bar{\varepsilon}_{n+1}:=F(x_{0},\bar{x}_{0},y_{0},\bar{y}_{0},\varepsilon_{n+1}).
$$ On the complement of $J_{0}$,
set $\bar{\varepsilon}_{n+1}:=g_0(x_{0},\bar{x}_{0},y_{0},\bar{y}_{0},U_{0})$. 
If $(x_0,y_0)\neq (\bar{x}_0,\bar{y}_0)$ but some of $x_0,\bar{x}_0,y_0,\bar{y}_0$ is outside $B_d(K_{0})$ then we just take 
$\bar{\varepsilon}_{n+1}:=\varepsilon_{n+1}$.
This definition
ensures that $\bar{\varepsilon}_{n+1}$ has the same law as $\varepsilon_{n+1}$.

Indeed, apply Lemma \ref{lemmacska} below
with the choice $$W:=(\varepsilon_{n+1},\hat{U}_{0}),\ Q=g_{0}(x_{0},y_{0},\bar{x}_{0},\bar{y}_{0},U_{0}),$$
and 
$A:=\{(e_{1},\hat{u}):\hat{u}\leq \eta_{0}, |e_{1}|\leq 1\}$, 
$$
G(e_{1},\hat{u}):=F(x_{0},y_{0},\bar{x}_{0},\bar{y}_{0},e_{1})1_{A}+e_{1}1_{(\mathbb{R}^{d}\times [0,1])\setminus A},$$
it is independent of the past and $X_{n+1} \!=\! \bar{X}_{n+1}$ holds on $J_{0}$.

Now, based on the above discussion, let us define {rigorously} the function $S_{0}$ that will be used for coupling.
If $(x_{0},y_{0})=(\bar{x}_{0},\bar{y}_{0})$ then we set $$
S_{0}(e_1,u,\hat{u},x_{0},y_{0},\bar{x}_{0},\bar{y}_{0}):=e_1.
$$ 

If $(x_{0},y_{0})\neq (\bar{x}_{0},\bar{y}_{0})$, $x_{0},y_{0},\bar{x}_{0},\bar{y}_{0}\in B_d(K_{0})$,
$e_1\in B_d(1)$ and
$\hat{u}\leq \eta_{0}$ then
$$S_0(e_1,u,\hat{u},x_{0},y_{0},\bar{x}_{0},\bar{y}_{0}):=F(x_{0},y_{0},\bar{x}_{0},\bar{y}_{0},e_1).
$$

If either $\hat{u}>\eta_{0}$ or $|e_{1}|>1$ then
$$
S_0(e_1,u,\hat{u},x_{0},y_{0},\bar{x}_{0},\bar{y}_{0}):=g(x_{0},\bar{x}_{0},y_{0},\bar{y}_{0},u).
$$

If $(x_{0},y_{0})\neq (\bar{x}_{0},\bar{y}_{0})$ and some of $x_{0},y_{0},\bar{x}_{0},\bar{y}_{0}$ fails to
lie in $B_d(K_{0})$
then we set again
$$S_0(e_1,u,\hat{u},x_{0},y_{0},\bar{x}_{0},\bar{y}_{0}):=e_{1}.$$

We now go on to time $n+1$, to the second step of the coupling procedure where $\bar{\varepsilon}_{n+2}$ will be
constructed. If $y_{0},\bar{y}_{0}\in B_{d}(K_{0})$, $|e_{1}|\leq 1$ and $\hat{u}\leq \eta_{0}$
then $$
|S_0(e_{1},u,\hat{u},x_0,\bar{x}_0,y_0,\bar{y}_0)|\leq c_{\flat}(K_{0})
$$
hence both processes can be bounded via
\begin{align*}
  &|(I-hA)\bar{y}_{0}+S_{0}(e_{1},u,\hat{u},x_0,\bar{x}_0,y_0,\bar{y}_0,e_1)|\\
  \leq~&\tilde{\delta}_{2}|\bar{y}_{0}| + c_{\flat}(K_{0}) \leq K_{0}+c_{\flat}(K_{0})+1=:K_{1},\\
&|(I-hA){y}_{0}+e_{1}|\leq \tilde{\delta}_{2}|{y}_{0}|+1\leq K_{0}
         +1\leq K_{1}.
\end{align*}
In the heuristic picture this means that on $J_{0}$ we have $Y_{n+1},\bar{Y}_{n+1}\in B_{d}(K_{1})$.

Consider for $y_1,\bar{y}_1,e_2\in\mathbb{R}^d$ (they are the frozen values of $Y_{n+1},\bar{Y}_{n+1}$, $\varepsilon_{n+2}$
in the heuristic picture),
$$
H(y_{1},\bar{y}_{1},e_2):=(I-hA)y_1+e_2-(I-hA)\bar{y}_{1}.
$$
Again, setting $\bar{e}_2$ to be $H(y_{1},\bar{y}_{1},e_2)$ will ensure that $Y_{n+2}=\bar{Y}_{n+2}$,
which, together with $X_{n+1}=\bar{X}_{n+1}$ will entail $X_{n+2}=\bar{X}_{n+2}$.

Clearly, the derivative of $H$ in its third variable is constant identity.
By boundedness of $\sigma$ and local boundedness of $\sigma^{-1}$, there is a constant $c_\natural(K_1)$ with
\begin{equation}\label{dorh}
|H(y_{1},\bar{y}_{1},e_2)|\leq c_\natural(K_1)
\end{equation}
for all $y_{1},\bar{y}_{1}\in B_d(K_{1})$ and $e_2\in B_d(1)$.

Let $y_{1},\bar{y}_{1}\in B_d(K_{1})$.
Let $\kappa_1(A):=\eta_1 \mathrm{Leb}_{d}(A\cap B_d(1))$ with some $\eta_1$ that will be chosen later.
Let $\bar{\kappa}_1(y_{1},\bar{y}_{1},\cdot)$ denote the image of $\kappa_1$ by the mapping $H(y_{1},\bar{y}_{1},\cdot)$.
This measure is concentrated on
$B_{d}(c_\natural(K_{1}))$, and its density is constant $\eta_{1}$ on the set $H(B_{d}(1))$ (and zero outside).

Now choose $\eta_1$ so small that
\begin{equation}\label{vo1}
\eta_1\leq \lambda_{c_\natural(K_{1})}/2\leq\inf_{\bar{e}_2\in B_d(c_{\natural}(K_{1}))}f(\bar{e}_2)/2.
\end{equation}

It follows from \eqref{vo1} that $$
\upsilon_1(y_{1},\bar{y}_{1}, A):=\frac{\int_A [f(\bar{e}_2)-\eta_{1}1_{H(B_{d}(1))}]\, d\bar{e}_2}
{\int_{\mathbb{R}^{d}} [f(\bar{e}_2)-\eta_{1}1_{H(B_{d}(1))}]\, d\bar{e}_2},\ A\in\mathcal{B}(\mathbb{R}^{d}),
$$
defines a probabilistic kernel. Let $(U_{1},\hat{U}_{1})$ be uniform on $[0,1]^{2}$. Analogously to Lemma \ref{consu}, there is 
$g_{1}:B_{d}(K_{1})^{2}\times [0,1]\to \mathbb{R}^{d}$
such that $g_{1}(y_{1},\bar{y}_{1},U_{1})$ has the same law as $\upsilon_{1}(y_{1},\bar{y}_{1},\cdot)$.

We describe how this step will appear in the heuristic description: if $y_{1}=\bar{y}_{1}$ then 
$\bar{\varepsilon}_{n+2}:=\varepsilon_{n+2}$;
if $y_1,\bar{y}_1\in B_d(K_{1})$ but $y_{1}\neq \bar{y}_{1}$, on the event
$J_1:=\{|\varepsilon_{n+2}|\leq 1\}\cap \{\hat{U}_1\leq \eta_1\}$ then
we define $\bar{\varepsilon}_{n+2}:=H(y_1,\bar{y}_1,\varepsilon_{n+2})$. On the complement of $J_1$,
set $\bar{\varepsilon}_{n+2}:=g_1(y_1,\bar{y}_1,U_1)$. If $y_{1}\neq \bar{y}_{1}$ but either $y_{1}$ or $\bar{y}_{1}$
is outside $B_d(K_{1})$ then we just take 
$\bar{\varepsilon}_{n+2}:=\varepsilon_{n+2}$.

This definition
ensures that $\bar{\varepsilon}_{n+2}$ has the same law as $\varepsilon_{n+2}$ and,
if $X_n=x$, $\bar{X}_n=\bar{x}$, $Y_n=y$, $\bar{Y}_n=\bar{y}$ then $X_{n+2}=\bar{X}_{n+2}$ and
$Y_{n+2}=\bar{Y}_{n+2}$ \emph{both} hold on the event $J_0\cap J_{1}$ which has probability 
$\eta_0\eta_1 P(|\varepsilon_1|\leq 1,|\varepsilon_2|\leq 1)$.

Let us define our second coupling function $S_{1}$:
if $y_{1}=\bar{y}_{1}$ then we set $$S_{1}(e_2,u,\hat{u},y_{1},\bar{y}_{1}):=e_2.$$ 

If $y_{1}\neq \bar{y}_{1}$ and $y_{1},\bar{y}_{1}\in B_d(K_{1})$ then for $e_2\in B_d(1)$ and
$\hat{u}\leq \eta_{1}$ set
$$S_1(e_2,u,\hat{u},y_{1},\bar{y}_{1}):=H(y_{1},\bar{y}_{1},e_2).$$

If $\hat{u}>\eta$ or $e_{2}\notin B_{d}(1)$ then
$$S_1(e_2,u,\hat{u},y_{1},\bar{y}_{1}):=g_{1}(y_{1},\bar{y}_{1},u).$$

If $y_{1}\neq \bar{y}_{1}$ and $y_{1}$ or $\bar{y}_{1}$ are not in $B_d(K_{1})$ then 
$S_1(e_2,u,\hat{u},y_{1},\bar{y}_{1}):=e_2$.
\subsection{Coupling of two processes}\label{too}

Let $(U_n,\hat{U}_n)$, $n\in\mathbb{N}$ be an i.i.d.\ sequence of
uniform random variables on $[0,1]^2$. Let us take $\varepsilon_j$, $j<0$ such that $(\varepsilon_j)_{j\in\mathbb{Z}}$
is an i.i.d. sequence, that is, we extend the noise sequence to a two-sided one.
Let $m'>m\geq 1$ be given. Let $(X_0,Y_0)$ be an initialization independent of $(\varepsilon_j)_{j\geq 1}$.


Let $(\bar{X}_{m-m'},\bar{Y}_{m-m'})$ be 
$\sigma(\varepsilon_j,j\leq m-m')$-measurable with the same law as $(X_0,Y_0)$.
We will carry out the construction of the process $(\bar{X},\bar{Y})$ in such a way that it will be suitably coupled with $(X,Y)$.

For $m-m'\leq k<0$ we simply set 
$\bar{Y}_{k+1}:=(I-hA)\bar{Y}_k+\varepsilon_{k+m'-m+1}$ and 
$\bar{X}_{k+1}:=\bar{X}_k+h\mu(\bar{X}_k)+\sigma(\bar{X}_k)\bar{Y}_{k+1}$.

Take $K_{0}:=2C_{3}$. For each even number $2k\geq 0$ we define $$
\bar{\varepsilon}_{2k+1}:=
S_{0}(\varepsilon_{2k+1},U_{2k+1},\hat{U}_{2k+1},X_{2k},Y_{2k},\bar{X}_{2k},\bar{Y}_{2k}),
$$
then $\bar{Y}_{2k+1}:=(I-hA)\bar{Y}_{2k}+\bar{\varepsilon}_{2k+1}$ and 
$\bar{X}_{2k+1}:=\bar{X}_{2k}+h\mu(\bar{X}_{2k})+\sigma(\bar{X}_{2k})\bar{Y}_{2k+1}$.
In the next step,
\begin{align*}
\bar{\varepsilon}_{2k+2} \!:=\!
S_{1}(\varepsilon_{2k+2},\! U_{2k+2},\! \hat{U}_{2k+2},\! X_{2k+1},\! Y_{2k+1},\! \bar{X}_{2k+1},\! \bar{Y}_{2k+1}),  
\end{align*}
$\bar{Y}_{2k+2}:=(I-hA)\bar{Y}_{2k+1}+\bar{\varepsilon}_{2k+2}$ 
and 
$\bar{X}_{2k+2}:=\bar{X}_{2k+1}+h\mu(\bar{X}_{2k+1})+\sigma(\bar{X}_{2k+1})\bar{Y}_{2k+2}$.

Clearly, $(\bar{X}_{j},\bar{Y}_{j})$ has the same law as 
$(X_{j+m'-m},Y_{j+m'-m})$ for all $j\geq 0$, in particular,
$(\bar{X}_{m},\bar{Y}_{m})$ has the same law as 
$(X_{m'},Y_{m'})$.

Define another filtration
$
\mathcal{F}_n':=\mathcal{F}_n\vee \sigma(U_j,\hat{U}_j;j\leq n). 
$
Notice that both $(X_n,Y_n)$ and $(\bar{X}_n,\bar{Y}_n)$
are adapted to $\mathcal{F}_n'$.

\begin{remark}\label{epsind}
{\rm Our construction ensures that both $\bar\varepsilon_{2k+1}$ and $\bar\varepsilon_{2k+2}$ have the same
law as $\varepsilon_{0}$, furthermore, $\bar\varepsilon_{2k+1}$ is independent of $\mathcal{F}_{2k}'$ and
$\bar\varepsilon_{2k+2}$ is independent of $\mathcal{F}_{2k+1}'$. Indeed, freezing the values of
the random variables in the definition of $\mathcal{F}_{2k}$ (resp. $\mathcal{F}_{2k+1}$), the
conditional law of $\bar\varepsilon_{2k+1}$ (resp. $\bar\varepsilon_{2k+2}$) will be that of $\varepsilon_{0}$.
In particular, 
all processes $(X_{n},Y_{n})$, $(\bar{X}_{n},\bar{Y}_{n})$ and their joint process $(X_{n},Y_{n},\bar{X}_{n},\bar{Y}_{n})$, $n\geq 0$
are Markovian with respect to the filtration $\mathcal{F}_{n}'$, $n\geq 0$.}
\end{remark}

Clearly, Lemma \ref{foster-lyapunov} applies to $(X_{n},Y_{n})$
with the choice $\mathcal{G}_n:=\mathcal{F}_n'$.
We need the analogue for
the process $(\bar{X}_n,\bar{Y}_n)$.

\begin{lemma}\label{3step} For $\delta_2$ and $C_2$ as in Lemma \ref{foster-lyapunov},
\begin{equation}\label{nemobar}
    \mathbb{E}[V(\bar{X}_{n+1},\bar{Y}_{n+1})\mid\mathcal{F}_n']\leq \delta_2 V(\bar{X}_{n},\bar{Y}_{n})+C_2.
    \end{equation}
    Furthermore, $C_{3}>0$ and $\delta_3$ as in Lemma \ref{foster-lyapunov}, on the event
$\{|X_n|+B|Y_n|+|\bar{X}_{n}|+B|\bar{Y}_{n}|>2C_3\}$
\begin{equation}
    \label{eq:lyap_contraction1}
    \begin{aligned}
      &\mE_n[V(X_{n+1},Y_{n+1})+V(\bar{X}_{n+1},\bar{Y}_{n+1})\mid\mathcal{F}_{n}']\\
      \leq~&\delta_3 [V(X_{n},Y_{n}) +V(\bar{X}_{n},\bar{Y}_{n})].
    \end{aligned}    
\end{equation}    
\end{lemma}
\begin{proof} The construction of $(\bar{X}_{n},\bar{Y}_{n})$ was done in such a way that it is $\mathcal{F}_{n}'$-measurable
and $\bar{\varepsilon}_{n+1}$ is independent of $\mathcal{F}_{n}'$, recall Remark \ref{epsind}. So we may apply Lemma 
\ref{foster-lyapunov} to
$(\bar{X}_{n},\bar{Y}_{n})$ instead of $(X_{n},Y_{n})$ with the choice $\mathcal{G}_{n}:=\mathcal{F}_{n}'$.

As in Lemma \ref{foster-lyapunov}, we can satisfy 
\begin{align*}
  &\delta_2 [V(X_{n},Y_{n})+V(\bar{X}_{n},\bar{Y}_{n})]+2C_2\\
  \leq ~&\delta_3 
[V(X_{n},Y_{n})+V(\bar{X}_{n},\bar{Y}_{n})]  
\end{align*}
once $V(X_{n},Y_{n})+V(\bar{X}_{n},\bar{Y}_{n})\geq 2C_2/(\delta_3-\delta_2)=2C_{3}$.
\end{proof}

The lemma above helps us to control return times to a suitably chosen compact set. 

\begin{lemma}
\label{lemmataud} 
Let us define $D= \{(x,y)\in \mathbb{R}^{2d}: |x|+B|y|\leq 2C_3\}$ 
Consider the processes $(X_n,Y_n)$ and $(\bar X_n, \bar Y_n)$, and define the double hitting time of $D$, the stopping time when both processes simultaneously reach $D$ as
$$
   \tau:=\inf\{n\geq 0: X_n,Y_n,\bar{X}_n,\bar{Y}_n\in D\}.
   $$
 %
   Then there exists $C_4>0$ so that with $\delta_3$ as in Lemma \ref{foster-lyapunov},
   \begin{equation}\label{taud}
    \mP(\tau > n)\leq C_4\delta_3 ^n \mathbb{E}[V(X_0,Y_0)+V(\bar X_0,\bar Y_0)].
\end{equation}
\end{lemma}
\begin{proof}
    For convenience denote $V_n = |X_n|+B|Y_n|$, $\bar V_n = |\bar X_n|+B|\bar Y_n|$.
%
We define the stopped process $M_n$ by setting $M_0 = V_0+\bar V_0$ and for $n\geq 0$ let
\[
M_{n+1} = \begin{cases}
    \alpha^{n+1} (V_{n+1}+\bar V_{n+1}) & \text{if }\tau>n,\\
    M_n & \text{if } \tau\leq n.
\end{cases}
\]
Here $\alpha=\delta_3^{-1}>1$ is chosen.
$\mathbb{E}_{m}[\cdot]$ will be an abbreviation
for $\mathbb{E}[\,\cdot\mid \mathcal{F}_{m}']$.
We will show this is a supermartingale. First, the one step conditional expectation takes the form
\begin{align*}
\mE_n[M_{n+1}] &= \mE_n[1_{\{\tau>n\}}\alpha^{n+1} (V_{n+1}+\bar V_{n+1}) + 1_{\{\tau\leq n\}}M_n]\\ 
&= 1_{\{\tau>n\}}\alpha^{n+1} \mE_n[V_{n+1}+\bar V_{n+1}] + 1_{\{\tau\leq n\}}M_n.
\end{align*}
Not having reached $\tau$ implies $V_n+\bar V_n\geq 2C_3$, 
thus by Lemma \ref{3step},
\begin{align*}
  &1_{\{\tau>n\}}\alpha^{n+1} \mE_n[V_{n+1}+\bar V_{n+1}]\\
  \leq~ &1_{\{\tau>n\}}\alpha^{n+1} \delta_3(V_{n}+\bar V_{n}) = 1_{\{\tau>n\}}\alpha \delta_3 M_n.
\end{align*}
Combining with the above and recalling $\alpha\delta_3=1$ we get
\[
\mE_n[M_{n+1}]\leq M_n.
\]

Now to use this for a bound on the hitting time $\tau$, we can combine its 
implication on the Lyapunov function with the supermartingale property and Markov' inequality to obtain
\begin{align*}
\mP(\tau>n)&= \mP(\tau>n,~V_n+\bar V_n\geq 2C_3)\\ &\leq
\mP(M_n\geq 2\alpha^n C_3) \leq \frac{\mathbb{E}[M_n]}{2\alpha^n C_3}\leq \frac{\mathbb{E}[V_0+\bar V_0]}{2\alpha^n C_3}.
\end{align*}
This confirms the exponential decay for the simultaneous hitting time as stated. 
Indeed $\delta_3 = \alpha^{-1}$ appears as the geometric factor, with a constant $C_4:=1/2C_3$.
\end{proof}

The simple fact below was used in this subsection.

\begin{lemma}\label{lemmacska}
Let $W\in \mathbb{R}^{d}\times [0,1]$ and $Q\in\mathbb{R}^{d}$ be independent random variables, 
$G:\mathbb{R}^{d}\times [0,1]\to\mathbb{R}^{d}$ a measurable function,
$A\in\mathcal{B}(\mathbb{R}^{d}\times [0,1])$ with $\mathbb{P}(W\notin A)>0$.
Assume also that
$\mathrm{Law}(G(W)|W\notin A)=\mathrm{Law}(Q)$. 
Then 
$$
\mathrm{Law}(G(W)1_{\{W\in A\}}+Q1_{\{W\notin A\}})=\mathrm{Law}(G(W)).
$$
\end{lemma}
\begin{proof}
Indeed, for any $H\in\mathcal{B}(\mathbb{R}^{d})$,
\begin{align*}
&\mathbb{P}(G(W)1_{\{W\in A\}}+Q1_{\{W\notin A\}}\in H)\\
=&\mathbb{P}(G(W)\in H, W\in A)+
    \mathbb{P}(Q\in A, W\notin A)\\
=&\mathbb{P}(G(W)\in H| W\in A)\mathbb{P}(W\in A) +
\mathbb{P}(Q\in A)\mathbb{P}(W\notin A)\\
=&\mathbb{P}(G(W)\in H| W\in A)\mathbb{P}(W\in A)\\
&+\mathbb{P}(G(W)\in H| W\notin A)\mathbb{P}(W\notin A)\\
=&\mathbb{P}(G(W)\in H).	
\end{align*}
\end{proof}

\subsection{Remaining arguments}\label{ra}

We will now track returns of both processes $(X_{n},Y_{n})$, $(\bar{X}_{n},\bar{Y}_{n})$ to the compact set $D$. 
Set $\tau_0:=0$ and for $k\in\mathbb{N}$ define
$$
\tau_{k+1}:=\inf\{l\geq \tau_k+3: ({X}_l,{Y}_l)\in D, 
(\bar{X}_l,\bar{Y}_l)\in D\}.
$$
Remember that we made the choice $K_{0}:=2C_{3}$ earlier, hence $({X}_l,{Y}_l)\in D$, 
$(\bar{X}_l,\bar{Y}_l)\in D$ and $B\geq 1$ imply $${X}_l,{Y}_l,\bar{X}_l,\bar{Y}_l\in B_{d}(K_{0}).$$

Remember that $\gamma:=\eta_{0}\eta_{1}P^{2}(|\varepsilon_{0}|\leq 1)$ is the probability
of $\{\hat{U}_{0}\leq \eta_{0},\hat{U}_{1}\leq \eta_{1},|\varepsilon_{n+1}|\leq 1,|\varepsilon_{n+2}|\leq 1\}$
in the construction of Subsection \ref{loco} above and that at every $\tau_{k}$ we couple the two copies of processes with this probability.

We distinguish two cases, whether $\tau_{k-1}$ is even or odd. Let us define the events
\begin{align*}
H&:=\{\tau_{k-1}\mbox{ is even}\},\\
R_n &:=\{ \hat{U}_{n+1}\leq \eta_{0},\hat{U}_{n+2}\leq \eta_{1}, |\varepsilon_{n+1}|\leq 1,|\varepsilon_{n+2}|\leq 1\}.
\end{align*}
Then
on $H$,  
\begin{align}\nonumber
 & \mathbb{P}(X_{\tau_{k-1}+3}\neq\bar{X}_{\tau_{k-1}+3}\mbox{ or }Y_{\tau_{k-1}+3}
\neq\bar{Y}_{\tau_{k-1}+3}\mid\mathcal{F}_{\tau_{k-1}})\\
\leq~&\nonumber  \mathbb{P}(X_{\tau_{k-1}+2}\neq\bar{X}_{\tau_{k-1}+2}\mbox{ or }Y_{\tau_{k-1}+2}\neq\bar{Y}_{\tau_{k-1}+2}
\mid\mathcal{F}_{\tau_{k-1}})\\
\leq~&\nonumber 1_{\{
(X_{\tau_{k-1}}\neq\bar{X}_{\tau_{k-1}}\mbox{ or }Y_{\tau_{k-1}}\neq\bar{Y}_{\tau_{k-1}})\}}(1-\mathbb{P}(R_{\tau_{k-1}}\mid\mathcal{F}_{\tau_{k-1}}'))\\
=~& \label{olt1} 1_{\{
(X_{\tau_{k-1}}\neq\bar{X}_{\tau_{k-1}}\mbox{ or }Y_{\tau_{k-1}}\neq\bar{Y}_{\tau_{k-1}})\}}(1-\gamma).
\end{align}
The first inequality here follows from the fact that if the two processes coincide
at time $\tau_{k-1}+2$ then they continue to be equal at $\tau_{k-1}+3$. The second
inequality follows from the construction of the coupling at $\tau_{k-1}$: we use $S_{0}$ (resp. $S_{1}$) and the random variables 
$\varepsilon_{\tau_{k-1}+1},U_{\tau_{k-1}+1},\hat{U}_{\tau_{k-1}+1}$ (resp. $\varepsilon_{\tau_{k-1}+2},,U_{\tau_{k-1}+2},
\hat{U}_{\tau_{k-1}+2}$) and the latter are easily seen to be independent of $\mathcal{F}_{\tau_{k-1}}'$ 
(resp. $\mathcal{F}_{\tau_{k-1}+1}'$), with the same joint law as $\varepsilon_{0},U_{0},\hat{U}_{0}$. Finally
the third equality comes from the previous observation and from the definition of $\gamma$.{}

On the event $\overline{H}$,
\begin{align}
 &\nonumber  \mathbb{P}(X_{\tau_{k-1}+3}\neq\bar{X}_{\tau_{k-1}+3}\mbox{ or }Y_{\tau_{k-1}+3}\neq\bar{Y}_{\tau_{k-1}+3}
\mid\mathcal{F}_{\tau_{k-1}})\\
\leq~&\nonumber \mathbb{P}(X_{\tau_{k-1}+1}\neq\bar{X}_{\tau_{k-1}+1}\mbox{ or }Y_{\tau_{k-1}+1}\neq\bar{Y}_{\tau_{k-1}+1}
\mid\mathcal{F}_{\tau_{k-1}})\times \\
\times~&\nonumber 
 (1-\mathbb{P}(R_{\tau_{k-1}+1}\mid\mathcal{F}_{\tau_{k-1}}'))\\
\leq~& \label{olt2} 1_{\{
(X_{\tau_{k-1}}\neq\bar{X}_{\tau_{k-1}}\mbox{ or }Y_{\tau_{k-1}}\neq \bar{Y}_{\tau_{k-1}})\}}
(1-\gamma).
\end{align}
The logic is the same just the order swapped, the first inequality holds by the construction of the couplings, the second holds thanks to coupled processes moving together.

\begin{proof}[of Theorem \ref{main}]
Inequalities \eqref{olt1} and \eqref{olt2} imply
\begin{align*}
 & \mathbb{P}\left(X_{\tau_k}\neq\bar{X}_{\tau_k}\mbox{ or }
Y_{\tau_k}\neq\bar{Y}_{\tau_k}\mid\mathcal{F}_{\tau_{k-1}}\right)\\
\leq ~& \mathbb{P}\left(X_{\tau_{k-1}+3}\neq\bar{X}_{\tau_{k-1}+3}\mbox{ or }Y_{\tau_{k-1}+3}\neq\bar{Y}_{\tau_{k-1}+3}
\mid\mathcal{F}_{\tau_{k-1}}\right)\\
\leq ~& (1-\gamma)1_{\{X_{\tau_{k-1}}\neq\bar{X}_{\tau_{k-1}}\mbox{ or }
Y_{\tau_{k-1}}\neq
\bar{Y}_{\tau_{k-1}}\}},
\end{align*}
and iterating this, finally we arrive at
\begin{align*}
\mathbb{P}(X_{\tau_k}  \!\neq\! \bar{X}_{\tau_k}\mbox{ or }Y_{\tau_k} \!\neq\! \bar{Y}_{\tau_k})
  \!\leq\! (1 \!-\!\gamma)^k
\mathbb{P}(X_{0} \!\neq\! \bar{X}_{0}\mbox{ or }Y_{0} \!\neq\! \bar{Y}_{0}).
  \end{align*}

We see that, for all $l$ and $k$,
\begin{equation}
    \label{almostcoupling}
\mathbb{P}(X_{l}\neq \bar{X}_{l}\mbox{ or }Y_{l}\neq\bar{Y}_{l})\leq 
(1-\gamma)^k+\mathbb{P}(\tau_k\geq l).
\end{equation}
By the strong Markov property, the random variables $\tau_k-\tau_{k-1}-3$, $k\geq 2$
are independent of $\mathcal{F}_{\tau_{k-1}}'$ and have the same law as
$\tau$ in Lemma \ref{lemmataud} (but starting the processes from $X_{0}=X_{\tau_{k-1}}$, 
$Y_{0}=Y_{\tau_{k-1}+3}$, $\bar{X}_{0}=\bar{X}_{\tau_{k-1}+3}$, $\bar{Y}_{0}=\bar{Y}_{\tau_{k-1}+3}$). 

Fix some $1<\delta_4$ with $\delta_4\delta_3<1$.
First, by Lemmas \ref{lemmataud}, \ref{moments}
\begin{eqnarray*}
&& 
\mE[\delta_4^{\tau_1}]\leq \sum_{k=1}^{\infty}\mathbb{P}(\tau_{1}>k-1)\delta_{4}^{k}\\
&\leq&{}
\sum_{k=1}^{\infty}C_{4}\delta_{3}^{k-1}
\delta_{4}^{k}\mathbb{E}[V(X_{0},Y_{0})+V(\bar{X}_{0},\bar{Y}_{0})]\\ &\leq&
C_{4}\frac{2\delta_{4}}{1-\delta_{3}\delta_{4}}C_{5}=:C_{6}<\infty,
\end{eqnarray*} 
since $\mathrm{Law}(\bar{X}_{0},\bar{Y}_{0})=\mathrm{Law}(X_{m},Y_{m})$ by construction and $C_{5}$ provides an
upper bound for both $\mathbb{E}[V(X_{m},Y_{m})]$ and $\mathbb{E}[V(X_{0},Y_{0})]$; also $\delta_{4}$
was chosen to satisfy $\delta_{3}\delta_{4}<1$.

Before estimating the exponential moments of $\tau_{k}$ for $k>1$, we need 
an auxiliary calculation.
%
Using Lemmas \ref{foster-lyapunov} and \ref{3step} and the trivial bound on $V$ at the hitting times,
\begin{align}
& \nonumber\mE[V(X_{\tau_{k-1}+3},Y_{\tau_{k-1}+3})\mid \mathcal{F}_{\tau_{k-1}}']\\
+&\nonumber \mE[V(\bar X_{\tau_{k-1}+3},\bar Y_{\tau_{k-1}+3})\mid \mathcal{F}_{\tau_{k-1}}']\\
&\leq\nonumber \delta_2^3V(X_{\tau_{k-1}},Y_{\tau_{k-1}})+
\delta_2^3V(\bar X_{\tau_{k-1}},\bar Y_{\tau_{k-1}})+\\
&+\nonumber 2(\delta_2^2+\delta_2+1)C_2\\
&\leq 2\delta_2^3C_3+ 2(\delta_2^2+\delta_2+1)C_2=:
C_7 <\infty.\label{innerpart}    
\end{align}

Now we will express the exponential moment of $\tau_k$, $k>1$. Notice that
\begin{align}
\label{taukmoment}
    \mE[\delta_4^{\tau_k}]&= 
    \mE[\delta_4^{\tau_{k-1}}
    \mE[\delta_4^{\tau_k-\tau_{k-1}}\mid \mathcal{F}_{\tau_{k-1}}']].
\end{align}
%
%
This implies, using \eqref{innerpart},
  \begin{align}
&\mE[\delta_4^{\tau_k-\tau_{k-1}}\mid \mathcal{F}_{\tau_{k-1}}'] \leq
\sum_{n=3}^\infty \delta_4^n \mP({\tau_k \!-\! \tau_{k-1}} \!>\! n \!-\! 1\mid \mathcal{F}_{\tau_{k-1}}')\nonumber\\
=~&\sum_{n=3}^\infty \delta_4^n \mP({\tau_k-\tau_{k-1}}-3>n-4\mid \mathcal{F}_{\tau_{k-1}}')  \label{iteri} \\
\leq~& \delta_{4}^{3}+\sum_{n=4}^\infty \delta_4^n C_4\delta_3^{n-4} C_7 =
\delta_{4}^{3}+\frac{C_{4}C_{7}\delta_{4}^{4}}{1-\delta_{3}\delta_{4}}=:C_{8}.    \nonumber
  \end{align}
%
%
Set $L:=\max\left\{C_{6},C_{8}\right\}$.
Using \eqref{taukmoment}, \eqref{iteri} iteratively, we get
\[
\mE[\delta_4^{\tau_k}]\leq L^{k-1}\mE[\delta_4^{\tau_1}]\leq L^k.
\]

Turning back to \eqref{almostcoupling}, for any $l$, set $k=\lfloor \beta l\rfloor$, with $\beta<1$ to be chosen. 
Then by a Chernoff bound, we have
\[
\mathbb{P}(X_{l}\neq \bar{X}_{l}\mbox{ or }Y_{l}\neq\bar{Y}_{l})\leq 
(1-\gamma)^{\lfloor \beta l\rfloor}+L^{\lfloor \beta l\rfloor}\delta_4^{-l}.
\]
For $\beta$ small enough, both terms are exponentially decaying. 




Choosing $l=m$, we conclude that for appropriate $\rho<1$ and $C_{9}>0$
\begin{equation}
  \label{eq:cauchy}
\mathbb{P}(\bar{X}_{m},\bar{Y}_{m})\neq (X_m,Y_m))\leq C_{9}\rho^{m}
\end{equation}
which implies that the sequence $\mathrm{Law}(X_{m},Y_{m})$ is Cauchy in total variation norm and relying on our coupling estimates its
limit $\Pi_{*}$ satisfies
\begin{align*}
  & d_{TV}(\Pi_{*},\mathrm{Law}(X_{m},Y_{m}))\\
  =~&\lim_{m'\to\infty}d_{TV}(\mathrm{Law}(X_{m'},Y_{m'}),\mathrm{Law}(X_{m},Y_{m}))
\leq C_{9}\rho^{m}.
\end{align*}
 (Note that, in fact, for each $m'$ a copy of $(\bar{X}_{n},\bar{Y}_{n})$, $n\in\mathbb{N}$
needs to be constructed but each of these satisfy \eqref{eq:cauchy}.)

The first statement of Theorem is thus proved.
It follows trivially that $\Pi_{*}$ is the unique stationary measure
with finite first moment.
\end{proof}

\section{Conclusion}\label{conclu}

Open questions of varying generality remain. One may look at $Y_{n}$ which has the dynamics of
a discretized ergodic diffusion. Our arguments carry over to this case, too, with suitable
modifications. 

It would be desirable to formulate abstract ergodicity conditions on
a general state space for dynamical systems driven by a Markovian process.  
By such tools one could study certain specific models, e.g.\ the waiting time in queuing models with dependent 
service and/or interarrival times, see e.g.\ \cite{attila3}. 
These extensions are also left for future research.

\end{document}